\documentclass[11pt]{article}

\usepackage{amssymb,amsmath, amsthm}
\usepackage{verbatim,mathrsfs,bbm}
\usepackage{bbold,mathbbol}
\usepackage[colorlinks,citecolor=red]{hyperref}
\usepackage{color}
\usepackage{young}
\usepackage[english]{babel}
\usepackage{tikz}

\usepackage{fullpage}

\usepackage{authblk}

\newtheorem{Theorem}{Theorem}
\newtheorem{Lemma}[Theorem]{Lemma}
\newtheorem{Example}[Theorem]{Example}
\newtheorem{Remark}[Theorem]{Remark}

\newtheorem{Corollary}[Theorem]{Corollary}
\newtheorem{Proposition}[Theorem]{Proposition}

\newcommand{\N}{\mathbb{N}}

\newcommand{\mf}[1]{\mathfrak{#1}} 
\newcommand{\mb}[1]{\mathbb{#1}} 

\newcommand{\mt}[1]{\text{#1}}

\usepackage[all]{xy}

\begin{document}
\title{Lexicographic shellability of the Bruhat-Chevalley order on fixed-point-free involutions}

\author[1]{Mahir Bilen Can}
\author[2]{Yonah Cherniavsky}
\author[3]{Tim Twelbeck}
\affil[1]{mahir.can@yale.edu, Tulane and Yale Universities, USA}
\affil[2]{yonahch@ariel.ac.il, Ariel University, Israel}
\affil[3]{ttwelbec@tulane.edu, Tulane University, New Orleans, USA}

\date{\today}
\maketitle

\begin{abstract}
The main purpose of this paper is to prove that the Bruhat-Chevalley ordering 
of the symmetric group when restricted to the fixed-point-free involutions forms
an $EL$-shellable poset whose order complex triangulates a ball.
Another purpose of this article is to prove that the Deodhar-Srinivasan poset
is a proper, graded subposet of the Bruhat-Chevalley poset on fixed-point-free involutions.
\end{abstract}


\vspace{.5cm}

{\em MSC-2010: 06A11, 14L30, 05E18}

\section{Introduction and preliminaries}

In this manuscript we are concerned with the interaction between two well known subgroups of the special linear group $\mt{SL}_{2n}$,
namely a Borel subgroup and a symplectic subgroup.
Without loss of generality, we choose the Borel subgroup $B$ to be the group of invertible upper triangular matrices,
and define the {\em symplectic group,} $\mt{Sp}_{2n}$ as the subgroup of fixed elements of the involutory automorphism
$\theta : \mt{SL}_{2n} \rightarrow \mt{SL}_{2n}$,
$\theta (g ) = J (g^{-1})^\top J^{-1}$,
where $J$ denotes the skew form
\begin{align*}
J =
\begin{pmatrix}
0 & id_n \\
-id_n & 0
\end{pmatrix},
\end{align*}
and $id_n$ is the $n\times n$ identity matrix.

It is clear that $B$ acts by left-multiplication on the symmetric space $\mt{SL}_{2n}/\mt{Sp}_{2n}$.
We investigate the covering relations of the poset $F_{2n}$ of inclusion relations among
the $B$-orbit closures.
It is known since the works of Beilinson-Bernstein \cite{BeilinsonBernstein81} and Vogan \cite{Vogan83}
that such inclusion posets have importance in the study of
discrete series representations of the real forms of semi-simple Lie groups.

To further motivate our discussion and help the reader to place our work appropriately
we look at a related situation.

It is well known that the symmetric group of permutation matrices, $S_m$ parametrizes the orbits of
the Borel group of upper triangular matrices $B\subset \mt{SL}_m$ in the flag variety $\mt{SL}_m/B$.
For $u\in S_m$, let $\dot{u}$ denote the right coset in $\mt{SL}_m/B$ represented by $u$.
The classical {\em Bruhat-Chevalley ordering} is defined by
$u \leq_{S_m} v  \iff B\cdot \dot{u} \subseteq \overline{B \cdot \dot{v}}$ for $u,v\in S_m$.

A permutation $u\in S_m$ is said to be an {\em involution}, if $u^2 = id$, or equivalently,
its permutation matrix is a symmetric matrix. We denote by $I_m$ the set of all involutions in $S_m$,
and consider it as a subposet of the Bruhat-Chevalley poset $(S_m,\leq_{S_m})$.
Let $m$ be an even number, $m=2n$.
An involution $x\in I_{2n}$ is called {\em fixed-point-free}, if the matrix of $x$ has no non-zero diagonal entries.
In [\cite{RS90}, Example 10.4], Richardson and Springer show that
there exists a poset isomorphism between the dual of $F_{2n}$ and the subposet of fixed-point-free involutions in $I_{2n}$.
Unfortunately, $F_{2n}$ does not form an interval in $I_{2n}$, hence it does not immediately inherit nice properties therein.
In fact, this is easily seen for $n=2$ from the Hasse diagram of $I_4$ in Figure \ref{fig:F4 in I4}, in which
the fixed point free involutions are boxed.

\begin{figure}[htp]
\begin{center}
\begin{tikzpicture}[scale=.4]

\node at (0,0) (c2) {$id$};

\node at (-7.5,5) (d1) {$(34)$};
\node at (0,5) (d2) {$(23)$};
\node at (7.5,5) (d3) {$(12)$};

\node at (-7.5,10) (e1) {$(24)$};
\node at (7.5,10) (e2) {$(13)$};
\node at (0,10) (e3) {$\framebox{(12)(34)}$};

\node at (-5,15) (f1) {$(14)$};
\node at (5,15) (f2) {$\framebox{(13)(24)}$};
\node at (0,20) (g) {$\framebox{(14)(23)}$};

\draw[-, very thick] (c2) to (d1);
\draw[-, very thick] (c2) to (d2);
\draw[-, very thick] (c2) to (d3);

\draw[-, very thick] (d1) to (e1);
\draw[-, very thick] (d1) to (e3);
\draw[-, very thick] (d2) to (e1);
\draw[-, very thick] (d2) to (e1);
\draw[-, very thick] (d2) to (e2);
\draw[-, very thick] (d2) to (e2);
\draw[-, very thick] (d3) to (e2);
\draw[-, very thick] (d3) to (e3);

\draw[-, very thick] (e1) to (f1);
\draw[-, very thick] (e1) to (f2);

\draw[-, very thick] (e3) to (f1);
\draw[-, very thick] (e2) to (f1);
\draw[-, very thick] (e2) to (f2);
\draw[-, very thick,color=blue] (e3) to (f2);

\draw[-, very thick] (f1) to (g);
\draw[-, very thick,color=blue] (f2) to (g);

\end{tikzpicture}
\end{center}
\label{fig:F4 in I4}
\caption{$F_4$ in $I_4$}
\end{figure}
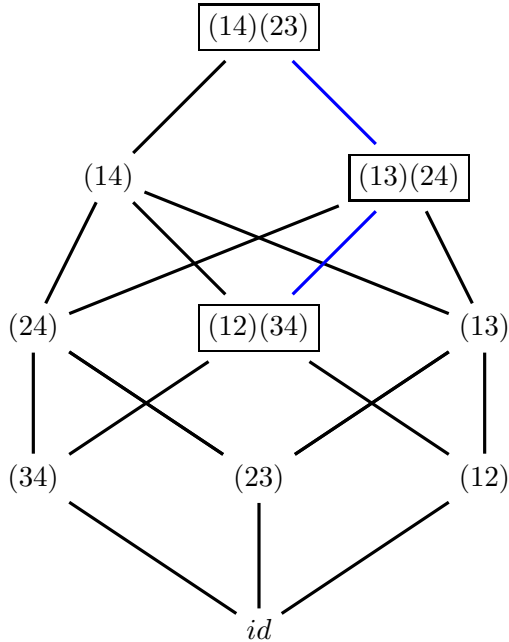


Let $\leq$ denote the restriction of the Bruhat-Chevalley ordering on $F_{2n}$.
Our first main result is that $(F_{2n},\leq)$ is ``$EL$-shellable,''
which is a property that is well known to be true for many other related posets.
See~\cite{Can12} (building on \cite{CR11},\cite{CanTwelbeck},\cite{Edelman81},\cite{Incitti04}, and~\cite{Proctor82}.

Recall that a finite graded poset $P$ with a maximum and a minimum element is called {\em $EL$-shellable},
if there exists a map $f=f_{\varGamma}: C(P) \rightarrow \varGamma$ from the
set of covering relations $C(P)$ of $P$ into a totally ordered set $\varGamma$ satisfying
\begin{enumerate}
\item in every interval $[x,y] \subseteq P$ of length $k>0$ there exists a unique saturated chain
$\mathfrak{c}:\ x_0=x < x_1 < \cdots < x_{k-1} < x_k=y$ such that the entries of the sequence
\begin{align}\label{JordanHolder}
f(\mathfrak{c}) = (f(x_0,x_1), f(x_1,x_2), \dots , f(x_{k-1},x_k))
\end{align}
are weakly increasing.

\item The sequence $f(\mathfrak{c})$ of the unique chain
$\mathfrak{c}$ from (1) is the lexicographically smallest among all sequences of the form
$(f(x_0,x_1'), f(x_1',x_2'), \dots , f(x_{k-1}',x_k))$, where $x_0 < x_1' < \cdots < x_{k-1}' < x_k$.

\end{enumerate}

For us, the {\em order complex} of a poset $P$ is the abstract simplicial complex $\Delta(P)$
whose simplicies are the chains in $\overline{P}=P - \{ \hat{0},\hat{1} \}$ (if the smallest 
element $\hat{0}$, and the largest element $\hat{1}$ are present in $P$). 
For an $EL$-shellable poset the order complex is shellable,
in particular it implies that $\Delta(P)$ is Cohen-Macaulay~\cite{Bjorner80}.
These, of course, are among the most desirable properties of a topological space.

As a corollary of our construction of the $EL$-labeling of $F_{2n}$, we prove a special case of a conjecture of
A. Hultman that the order complex of (the proper part of) $F_{2n}$ triangulates a ball of dimension $n^2-n-2$.
See Conjecture 6.3, \cite{H}. See \cite{Hultman05}, also.

Before we continue with explaining our other results, let us mention an important development 
which can be seen as a sequel to our work.
Let $n$ be an arbitrary positive integer.
Recall that the set of all $n\times n$ skew-symmetric matrices form the {\em special orthogonal Lie algebra}, $\mf{so}_{n}$.
The {\em congruence action} of $\mt{SL}_{n}$ on this Lie algebra is defined by $g\cdot A = (g^{-1})^\top A g^{-1}$.
In his Ph.D. thesis, the third author, by extending our labeling, proves that the inclusion poset of closures of the
Borel orbits in $\mf{so}_n$ (via the congruence action) is an $EL$-shellable poset.
For details, see Twelbeck's dissertation \cite{Twelbeck2013}.

In the literature there are different versions of lexicographic shellability.
A closely related notion with the same topological consequences as $EL$-shellability is called ``$CL$-shellability.''
We do not introduce its definition here, for more
we recommend the excellent monograph \cite{Wachs07} of Wachs who is one of the inventors of this notion.
It is known that $EL$-shellability implies $CL$-shellability,
however, whether the converse is true is an open problem for about thirty years.
In \cite{BW82}, Bj\"orner and Wachs show that Bruhat order on
all Coxeter groups, as well as on all sets of minimal-length coset representatives (quotients)
in Coxeter groups are ``dual $CL$-shellable.''
A decade after the introduction of $CL$-shellability, in \cite{Dyer93}, M. Dyer shows
that Bruhat order on all Coxeter groups and all quotients are $EL$-shellable.
Using Dyer's methods, in her 2007 Crelle paper \cite{Williams07}, L. Williams shows
that the poset of cells of a cell decomposition of the totally non-negative part of a flag variety
is $EL$-shellable. 

There are various directions that the results of \cite{BW82} are extended. For semigroups,
in \cite{Putcha02}, Putcha shows that ``$J$-classes in Renner monoids'' are $CL$-shellable.
In another direction, in Theorem~6.4 of~\cite{RV}, it is claimed that the Bruhat order 
on ``quasiparabolic'' sets in Coxeter groups is $CL$-shellable.
In particular, the fixed-point-free involutions form a quasiparabolic set in $S_{2n}$. 
However, it is pointed out to us by the referee of our paper and confirmed by one of the 
authors of~\cite{RV} that the proof of lexicographic shellability in Theorem~6.4 of~\cite{RV} 
seems to have a flaw and it is not obvious that it can be fixed. 
Most of the main results of~\cite{RV} are not affected by this.

One of the reasons the shellability of $F_{2n}$ is not considered before is that
there is a closely related $EL$-shellable partial order studied by Deodhar and Srinivasan in \cite{DeodharSrinivasan},
which was thought by several authors to be the same as Bruhat order on $F_{2n}$: 
for example, see page 248 (at the end of Section 2) in~\cite{Incitti04}, 
see also~\cite{Bona}, exercise 15 on page 307 and the corresponding note on page 312. 
As noticed in \cite{H} (without proofs) Deodhar-Srinivasan's poset differs from the Bruhat-Chevalley ordering on $F_{2n}$.
Here in our paper, we analyze the difference between these posets in more detail.
Let us mention also that, as it turns out, Deodhar and Srinivasan's poset is a particular case of a combinatorial
construction presented in~\cite{CanCherniavsky13}. See also~\cite{UV13}.

Obviously, every $x\in F_{2n}$ is expressible as a product of transpositions.
Indeed, let $i_1, \dots, i_n$ denote the list of all numbers from $\{1,\dots, 2n\}$ 
such that $x(i_r) > i_r$ for $r=1,\dots,n$.
Then $x = (i_1,x(i_1))(i_2,x(i_2))\cdots (i_n,x(i_n))$.
Note that disjoint cycles (hence transpositions) commute, therefore, 
to insist on the uniqueness of the expression, we require
that $i_1 < \dots < i_n$. In this case, by a change of notation, in place of $x$
we write $[i_1,x(i_1)][ i_2,x(i_2)]\cdots [i_n,x(i_n)]$.
Let $\tilde{F}_{2n}$ denote the set of all such unique ordered expressions, one for each $x\in F_{2n}$.

The partial ordering of \cite{DeodharSrinivasan}, which we call the {\em Deodhar-Srinivasan partial ordering}
and denote by $\leq_{DS}$, is defined as the transitive closure of the following relations.

$y=[c_1,d_1]\dots [c_n,d_n]\in \tilde{F}_{2n}$ is said to be greater than $x=[a_1,b_1]\cdots [a_n,b_n]\in \tilde{F}_{2n}$
in $\leq_{DS}$, if there exist $1 \leq i<j \leq n$ such that
\begin{enumerate}
\item $y$ is obtained from $x$ by interchanging $b_i$ and $a_j$, or
\item $y$ is obtained from $x$ by interchanging $b_i$ and $b_j$.
\end{enumerate}

A careful inspection of the Hasse diagrams of $(F_{2n},\leq)$ and $(\tilde{F}_{2n},\leq_{DS})$ 
reveals that these two posets are ``almost'' the same but different.
Our second main result is that the rank functions of these posets are the same, 
and furthermore, the latter is a graded subposet of the former.

The organization of our manuscript is as follows. 
In Section \ref{S:EL labeling}, we recall some known facts and study covering relations of $F_{2n}$.
In Section \ref{S:Main Theorem} we prove our first main result, 
and in Section \ref{S:comparison} we prove our second main result.
In Section \ref{S:ordercomplex} we show that $\Delta(F_{2n})$ triangulates 
a ball of dimension $n^2-n-2$.
We conclude our paper in Section \ref{son} with a short discussion of the various 
equivalent characterizations of the length function of $\ell_{F_{2n}}$.



Let $m$ be a positive integer. We denote the set $\{1,\dots,m\}$ by $[m]$.
In this paper, all posets are assumed to be finite and assumed to have
a minimal and a maximal element, denoted by $\hat{0}$ and $\hat{1}$, respectively.
Recall that in a poset $P$, an element $y$ is said to {\em cover} another element $x$, if
$x < y$ and if $x \leq z \leq y$ for some $z\in P$, then either $z=x$ or $z=y$.
In this case, we write $x \leftarrow y$.
Given $P$, we denote by $C(P)$ the set of all covering relations of $P$.

An (increasing) {\em chain} in $P$ is a sequence of distinct elements such that
$x=x_1 < x_2 < \cdots < x_{n-1} < x_n = y$.
A chain in a poset $P$ is called {\em saturated}, if it is of the form
$x=x_1 \leftarrow x_2 \leftarrow \cdots \leftarrow x_{n-1} \leftarrow x_n = y$.
A saturated chain in an interval $[x,y]$ is called {\em maximal}, if the end points of the chain are $x$ and $y$.
Recall also that a poset is called {\em graded} if all maximal chains between any two comparable elements $x \leq y$
have the same length.
This amounts to the existence of an integer valued function $\ell_P : P\rightarrow \N$ satisfying
\begin{enumerate}
\item $\ell_P (\hat{0}) = 0$,
\item $\ell_P (y) = \ell_P(x) +1$ whenever $y$ covers $x$ in $P$.
\end{enumerate}
$\ell_P$ is called the {\em length function} of $P$.
In this case, the length of $\hat{1}$ is called the {\em length} of the poset $P$.

\vspace{.5cm} 
\noindent 
\textbf{Acknowledgements.}
We express a deep gratitude to the referee for his very careful reading of our paper, 
his comments and suggestions which helped us very much to improve the paper, 
especially for pointing out that the lexicographic shellability result of our paper
was not established earlier.
The first and the third authors are partially supported by a 
Louisiana Board of Regents Research and Development Grant.
Authors would like to thank Michael Joyce, Roger Howe, Axel Hultman, 
Lex Renner, Yuval Roichman, and Bruce Sagan.

\section{$EL$-labeling}
\label{S:EL labeling}

\subsection{Incitti's $EL$-labeling}

For a permutation $\sigma \in S_n$, a {\em rise} of $\sigma$ is a pair of indices $1\leq i_1,i_2\leq n$ such that
$i_1 < i_2\  \text{and}\ \sigma(i_1)<\sigma(i_2)$.
A rise $(i_1,i_2)$ is called {\em free}, if there is no $k \in [n]$ such that
$i_1<k<i_2\ \text{and}\ \sigma(i_1)<\sigma(k)<\sigma(i_2)$. 
For $\sigma \in S_n$, define its {\em fixed point set, its exceedance set} and its {\em defect set} to be
\begin{align*}
I_f(\sigma) &=Fix(\sigma)=\{i \in [n]:\sigma(i)=i\}, \\
I_e(\sigma) &=Exc(\sigma)=\{i \in [n]:\sigma(i)>i\}, \\
I_d(\sigma) &=Def(\sigma)=\{i \in [n]:\sigma(i)<i\},
\end{align*}
respectively.
Given a rise $(i_1,i_2)$ of $\sigma$, its {\em type} is defined to be the pair $(a,b)$, if
$i_1 \in I_a(\sigma)$ and $i_2 \in I_b(\sigma)$, for some $a,b \in \{f,e,d\}$.
We call a rise of type $(a,b)$ an {\em $ab$-rise}.
On the other hand, two kinds of $ee$-rises have to be distinguished from each other;
an $ee$-rise is called {\em crossing}, if
$i_1<\sigma(i_1)<i_2<\sigma(i_2)$, and it is called {\em non-crossing}, if $i_1<i_2<\sigma(i_1)<\sigma(i_2)$.
The rise $(i_1,i_2)$ of an involution $\sigma \in I_n$ is called {\em suitable} 
if it is free and if its type is one of the following: $(f,f), (f,e), (e,f), (e,e), (e,d).$
A {\em covering transformation}, denoted $ct_{(i_1,i_2)}(\sigma)$, of a suitable rise $(i_1,i_2)$ of $\sigma$ is
the involution obtained from $\sigma$ by moving the 1's from the black dots to the white dots as depicted
in Table 1 of~\cite{Incitti04}.

Let $\leq$ denote the Bruhat-Chevalley ordering on the set of involutions of $S_{n}$. 
It is shown in~\cite{Incitti04} that if $\tau$ and $\sigma$ are from $I_n$, then
\begin{align*}
\tau \ \text{covers}\ \sigma\ \text{in}\ \leq\  \iff \ \tau = ct_{(i_1,i_2)}(\sigma),\ 
\text{for some suitable rise}\ (i_1,i_2)\ \text{of}\ \sigma.
\end{align*}
Let $\varGamma$ denote the totally ordered set $[n]\times [n]$ with respect to lexicographic ordering.
In the same paper, Incitti shows that the labeling defined by 
$f_\varGamma ((\sigma, ct_{(i_1,i_2)}(\sigma))) := (i_1,i_2) \in \varGamma$ 
is an $EL$-labeling, hence, $(I_n, \leq)$ is a lexicographically shellable poset.

\subsection{Covering transformations in $F_{2n}$}

Let $x$ and $y$ be two fixed-point-free involutions.
It follows from~\cite{H} (Theorem~4.6 and Example~3.4) 
that there exists a saturated chain between $x$ and $y$ that 
is entirely contained in $F_{2n}$. Notice that this fact can also be 
easily seen using the formulas for the length function of $F_{2n}$ 
and $I_{2n}$ presented in~\cite{Cherniavsky11} and~\cite{BC12}.
Therefore, its covering relations are among the covering relations of $I_{2n}$.
On the other hand, within $F_{2n}$ we use two types of covering transformations, only.
For convenience of the reader, we depict these moves in Figure \ref{Fs1} and Figure \ref{Fs2}.
These moves correspond to the items numbered 4 and 6 in Table 1 of~\cite{Incitti04}.

\begin{figure}[htp]
\begin{center}
\begin{tikzpicture}[scale=.75]

\tikzset{mymy/.style={circle,draw=blue, very thick, inner sep=.12em, fill=white }}
\tikzset{mymy0/.style={circle,draw=blue,very thick, fill=black, inner sep=.02em, minimum size=.4em}}
\node at (0,0) {$\longleftarrow$};

\begin{scope}[xshift=-4.25cm,yshift=0cm]
\node at (0,-3.75) {$\sigma$};
\draw[dotted, thick,black] (-2.5,2.5) -- (2.5,2.5) -- (2.5,-2.5) -- (-2.5,-2.5) -- (-2.5,2.5);
\draw[dotted, thick, black] (-2.5,2.5) -- (2.5,-2.5);
\filldraw[dotted,fill=gray!45,opacity=.95] (.5,.5) -- (1.5,.5) -- (1.5,1.5) -- (.5,1.5) -- (.5,.5) ;
\filldraw[dotted,fill=gray!45,opacity=.95] (-.5,-.5) -- (-1.5,-.5) -- (-1.5,-1.5) -- (-.5,-1.5) -- (-.5,-.5) ;

\node at (.5,.5) {0};
\node at (1.5,.5) {1};
\node at (1.5,1.5) {0};
\node at (.5,1.5) {1};
\node at (-.5,-.5) {0};
\node at (-1.5,-.5) {1};
\node at (-1.5,-1.5) {0};
\node at (-.5,-1.5) {1};
\node at (-2.75,1.5) {$i$};
\node at (-2.75,.5) {$j$};
\node at (-3,-.5) {$\sigma(i)$};
\node at (-3,-1.5) {$\sigma(j)$};
\node at (-1.5,2.85) {$i$};
\node at (-.5,2.85) {$j$};
\node at (.5,2.85) {$\sigma(i)$};
\node at (1.75,2.85) {$\sigma(j)$};
\end{scope}
\begin{scope}[xshift=4.5cm,yshift=0cm]
\node at (0,-3.75) {$\tau$};
\draw[dotted, thick,black] (-2.5,2.5) -- (2.5,2.5) -- (2.5,-2.5) -- (-2.5,-2.5) -- (-2.5,2.5);
\draw[dotted, thick, black] (-2.5,2.5) -- (2.5,-2.5);
\filldraw[dotted,fill=gray!45,opacity=.95] (.5,.5) -- (1.5,.5) -- (1.5,1.5) -- (.5,1.5) -- (.5,.5) ;
\filldraw[dotted,fill=gray!45,opacity=.95] (-.5,-.5) -- (-1.5,-.5) -- (-1.5,-1.5) -- (-.5,-1.5) -- (-.5,-.5) ;
\node at (.5,.5) {1};
\node at (1.5,.5) {0};
\node at (1.5,1.5) {1};
\node at (.5,1.5) {0};
\node at (-.5,-.5) {1};
\node at (-1.5,-.5) {0};
\node at (-1.5,-1.5) {1};
\node at (-.5,-1.5) {0};
\node at (-2.75,1.5) {$i$};
\node at (-2.75,.5) {$j$};
\node at (-3,-.5) {$\sigma(i)$};
\node at (-3,-1.5) {$\sigma(j)$};
\node at (-1.5,2.85) {$i$};
\node at (-.5,2.85) {$j$};
\node at (.5,2.85) {$\sigma(i)$};
\node at (1.75,2.85) {$\sigma(j)$};
\end{scope}
\end{tikzpicture}
\end{center}
\caption{(non-crossing) $ee$-rise for the covering $\tau \rightarrow \sigma$.}
\label{Fs1}
\end{figure}
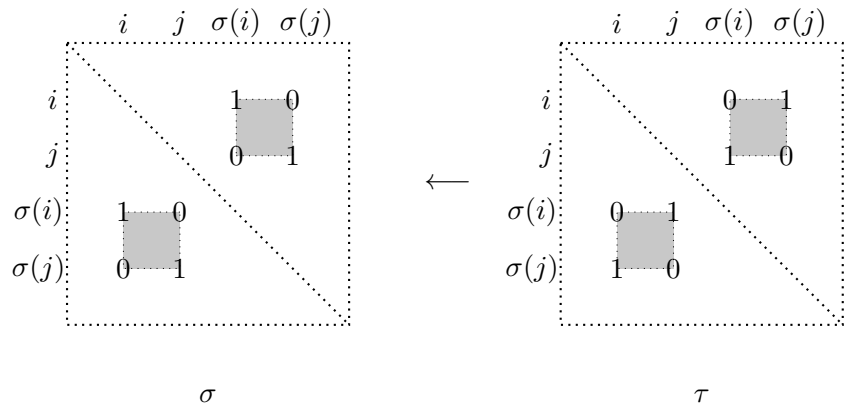

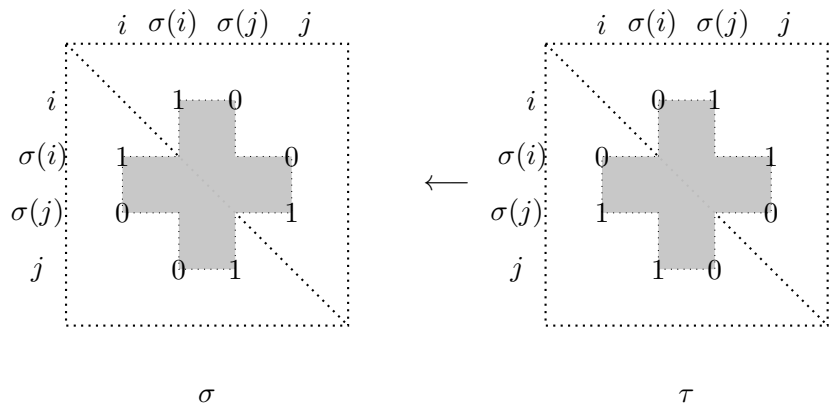
\begin{figure}[htp]
\begin{center}
\begin{tikzpicture}[scale=.75]
\tikzset{mymy/.style={circle,draw=blue, very thick, inner sep=.12em, fill=white }}
\tikzset{mymy0/.style={circle,draw=blue,very thick, fill=black, inner sep=.02em, minimum size=.4em}}
\node at (0,0) {$\longleftarrow$};
\begin{scope}[xshift=-4.25cm,yshift=0cm]
\node at (0,-3.75) {$\sigma$};
\draw[dotted, thick,black] (-2.5,2.5) -- (2.5,2.5) -- (2.5,-2.5) -- (-2.5,-2.5) -- (-2.5,2.5);
\draw[dotted, thick, black] (-2.5,2.5) -- (2.5,-2.5);
\filldraw[dotted,fill=gray!45,opacity=.95] (-.5,.5) -- (-.5,1.5) -- (.5,1.5) -- (.5,.5) -- (1.5,.5) -- (1.5, -.5)  -- (.5,-.5)
-- (.5,-1.5) -- (-.5,-1.5) -- (-.5, -.5) -- (-1.5,-.5) -- (-1.5, .5) -- (-.5,.5);
\node at (-.5,1.5) {1};
\node at (.5,1.5) {0};
\node at (1.5,.5) {0};
\node at (1.5,-.5) {1};
\node at (.5,-1.5) {1};
\node at (-.5,-1.5) {0};
\node at (-1.5,-.5) {0};
\node at (-1.5,.5) {1};
\node at (-2.75,1.5) {$i$};
\node at (-2.9,.5) {$\sigma(i)$};
\node at (-3,-.5) {$\sigma(j)$};
\node at (-3,-1.5) {$j$};
\node at (-1.5,2.85) {$i$};
\node at (-.6,2.85) {$\sigma(i)$};
\node at (.65,2.85) {$\sigma(j)$};
\node at (1.75,2.85) {$j$};
\end{scope}

\begin{scope}[xshift=4.25cm,yshift=0cm]
\node at (0,-3.75) {$\tau$};
\draw[dotted, thick,black] (-2.5,2.5) -- (2.5,2.5) -- (2.5,-2.5) -- (-2.5,-2.5) -- (-2.5,2.5);
\draw[dotted, thick, black] (-2.5,2.5) -- (2.5,-2.5);
\filldraw[dotted,fill=gray!45,opacity=.95] (-.5,.5) -- (-.5,1.5) -- (.5,1.5) -- (.5,.5) -- (1.5,.5) -- (1.5, -.5)  -- (.5,-.5)
-- (.5,-1.5) -- (-.5,-1.5) -- (-.5, -.5) -- (-1.5,-.5) -- (-1.5, .5) -- (-.5,.5);
\node at (-.5,1.5) {0};
\node at (.5,1.5) {1};
\node at (1.5,.5) {1};
\node at (1.5,-.5) {0};
\node at (.5,-1.5) {0};
\node at (-.5,-1.5) {1};
\node at (-1.5,-.5) {1};
\node at (-1.5,.5) {0};
\node at (-2.75,1.5) {$i$};
\node at (-2.9,.5) {$\sigma(i)$};
\node at (-3,-.5) {$\sigma(j)$};
\node at (-3,-1.5) {$j$};
\node at (-1.5,2.85) {$i$};
\node at (-.6,2.85) {$\sigma(i)$};
\node at (.65,2.85) {$\sigma(j)$};
\node at (1.75,2.85) {$j$};
\end{scope}
\end{tikzpicture}
\end{center}
\caption{$ed$-rise for the covering $\tau \rightarrow \sigma$.}
\label{Fs2}
\end{figure}

\section{Main Theorem}
\label{S:Main Theorem}

\begin{Theorem}\label{MainTheorem}
$F_{2n}$ is an $EL$-shellable poset.
\end{Theorem}

\begin{proof}

Let $x$ and $y$ be two fixed-point-free involutions from $F_{2n}$. 
Then there exists a saturated chain between $x$ and $y$ that is entirely contained in $F_{2n}$. 
Since lexicographic ordering is a total order on maximal chains, 
there exists a unique largest such chain. We denote it by
$
\mf{c}: x = x_1 < x_2 < \cdots < x_s = y.
$
The idea of the proof is showing that $\mf{c}$ is the unique decreasing chain and therefore
by switching the order of our totally ordered set $\mb{Z}^2$ obtaining the lexicographically 
smallest chain which is the unique increasing chain. See Figure \ref{fig:n=6} 
on page \pageref{fig:n=6} for an illustration.

\begin{figure}[htp]
\begin{center}
\begin{tikzpicture}[scale=.5]

\node at (0,0) (a) {$(12)(34)(56)$};

\node at (-5,5) (b1) {$(13)(24)(56)$};
\node at (5,5) (b2) {$(12)(35)(46)$};

\node at (-10,10) (c1) {$(14)(23)(56)$};
\node at (0,10) (c2) {$(13)(25)(46)$};
\node at (10,10) (c3) {$(12)(36)(45)$};

\node at (-10,15) (d1) {$(15)(23)(46)$};
\node at (0,15) (d2) {$(14)(25)(36)$};
\node at (10,15) (d3) {$(13)(26)(45)$};

\node at (-10,20) (e1) {$(15)(24)(36)$};
\node at (10,20) (e2) {$(14)(26)(35)$};
\node at (0,20) (e3) {$(16)(23)(45)$};

\node at (-5,25) (f1) {$(15)(26)(34)$};
\node at (5,25) (f2) {$(16)(24)(35)$};

\node at (0,30) (g) {$(16)(25)(34)$};

\node[color=blue] at (-3.5,2.5) {$(1,4)$};
\node[color=blue] at (3.5,2.5) {$(3,6)$};

\node[color=blue] at (-8.5,7.5) {$(1,2)$};
\node[color=blue] at (-3.5,7.5) {$(2,6)$};
\node[color=blue] at (3.5,7.5) {$(1,5)$};
\node[color=blue] at (8.5,7.5) {$(3,4)$};

\node[color=blue] at (-11,12.5) {$(1,6)$};
\node[color=blue] at (-7.5,13) {$(1,2)$};
\node[color=blue] at (-3.5,14) {$(2,6)$};
\node[color=blue] at (3.5,14) {$(1,5)$};
\node[color=blue] at (0,12.5) {$(1,6)$};
\node[color=blue] at (7,14) {$(2,4)$};
\node[color=blue] at (11,12.5) {$(1,6)$};

\node[color=blue] at (-11,17.5) {$(2,6)$};
\node[color=blue] at (-1.5,18.5) {$(1,4)$};
\node[color=blue] at (-8.5,18.5) {$(1,2)$};
\node[color=blue] at (8.5,18.5) {$(2,3)$};
\node[color=blue] at (1.5,18.5) {$(1,2)$};
\node[color=blue] at (11,17.5) {$(1,5)$};

\node[color=blue] at (-8,23) {$(2,3)$};
\node[color=blue] at (-5,21) {$(1,3)$};
\node[color=blue] at (-3,23.5) {$(1,3)$};
\node[color=blue] at (8.5,22.5) {$(1,2)$};
\node[color=blue] at (4.5,23.5) {$(2,5)$};

\node[color=blue] at (-3.5,27.5) {$(1,2)$};
\node[color=blue] at (3.5,27.5) {$(2,3)$};
\draw[-, very thick] (a) to (b1);
\draw[-, very thick,color=red] (a) to (b2);
\draw[-, very thick] (b1) to (c1);
\draw[-, very thick] (b1) to (c1);
\draw[-, very thick] (b1) to (c2);
\draw[-, very thick] (b2) to (c2);
\draw[-, very thick,color=red] (b2) to (c3);
\draw[-, very thick] (c1) to (d1);
\draw[-, very thick] (c1) to (d2);
\draw[-, very thick] (c2) to (d1);
\draw[-, very thick] (c2) to (d2);
\draw[-, very thick] (c2) to (d3);
\draw[-, very thick,color=red] (c3) to (d3);
\draw[-, very thick] (c3) to (d2);
\draw[-, very thick] (d1) to (e1);
\draw[-, very thick] (d1) to (e3);
\draw[-, very thick] (d2) to (e1);
\draw[-, very thick] (d2) to (e1);
\draw[-, very thick] (d2) to (e2);
\draw[-, very thick] (d2) to (e2);
\draw[-, very thick,color=red] (d3) to (e2);
\draw[-, very thick] (d3) to (e3);
\draw[-, very thick] (e1) to (f1);
\draw[-, very thick] (e1) to (f2);
\draw[-, very thick,color=red] (e2) to (f1);
\draw[-, very thick] (e2) to (f2);
\draw[-, very thick] (e3) to (f2);
\draw[-, very thick,color=red] (f1) to (g);
\draw[-, very thick] (f2) to (g);

\end{tikzpicture}
\end{center}
\label{fig:n=6}
\caption{Bruhat-Chevalley order on $\mt{SL}_6/\mt{Sp}_6$.}\label{fig:SL6/Sp6}
\end{figure}
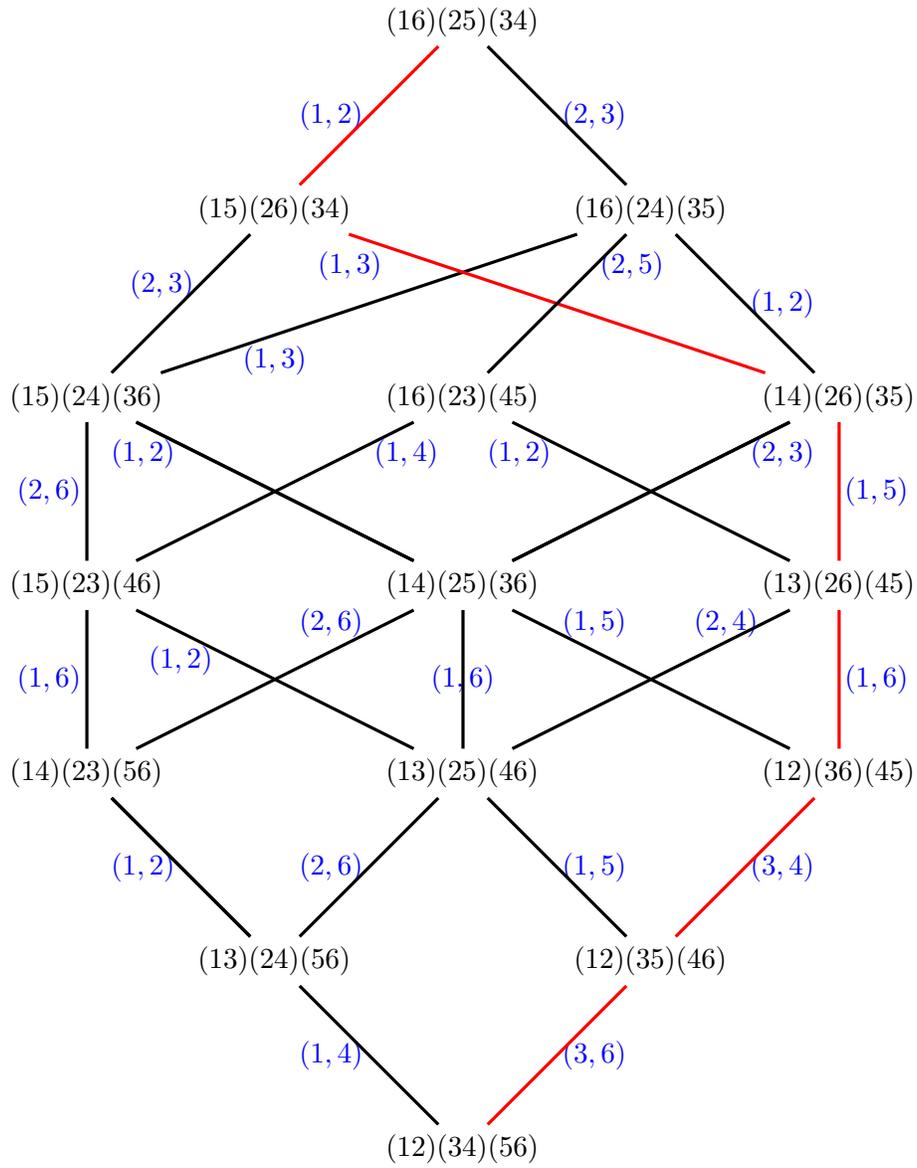

Towards a contradiction assume that $\mf{c}$ is not decreasing. 
Then, there exist three consecutive terms
\begin{equation*}
\sigma = x_{t-1} <  \tau = x_t <  \gamma = x_{t+1}
\end{equation*}
in $\mathfrak{c}$, such that $f((\sigma,\tau)) < f((\tau , \gamma))$.
We have 4 cases to consider.

        \begin{center}

        Case 1: $type(\sigma,\tau)=ee$, and $type(\tau , \gamma)=ee$.\\

        Case 2: $type(\sigma,\tau)=ed$, and $type(\tau , \gamma)=ed$.\\

        Case 3: $type(\sigma,\tau)=ee$, and $type(\tau , \gamma)=ed$.\\

        Case 4: $type(\sigma,\tau)=ed$, and $type(\tau , \gamma)=ee$.\\

        \end{center}

In each of these 4 cases, we either produce an immediate contradiction 
by showing that either the two moves are interchangeable
(hence $\mf{c}$ is not the largest chain), or we construct an element $z\in [x,y] \cap F_{2n}$ 
which covers $\sigma$, and such that $f((\sigma,z)) > f((\sigma,\tau))$. Since we assume that
$f(\mf{c})$ is the lexicographically largest, the existence of $z$ is a contradiction, too.
To this end, suppose that the label of the first move is $(i,j)$, and the second move is labeled by $(k,l)$.

\vspace{.5cm}

{\em Case 1:}

We begin with $\{ i,j,\sigma(i),\sigma(j)\} \cap \{ k,l \} = \emptyset$.

Assume for the moment that $k > j$. Then the covering transformations $(k,l)$ and $(i,j)$ 
are independent of each other.
Therefore, we assume that $i<k<j$. There are four cases; $\sigma(k) < j$,
$j<\sigma(k)<\sigma(i)$, $\sigma(i)<\sigma(k)<\sigma(j)$ or $\sigma(k)>\sigma(j)$.
In the first case $(k,\sigma(j))$ is an $ed$-rise for $\sigma$ with a label bigger than $(i,j)$.
This is a contradiction.
Similarly, in the second case, $(k,j)$ is an $ee$-rise for $\sigma$ with a bigger label than $(i,j)$.
The third case leads to a contradiction, because in that case $(i,j)$ is not a suitable rise in $\sigma$.
Finally, in the fourth case the two covering relations $(k,l)$ and $(i,j)$ are independent of each other.

Next we assume that $\{ i,j,\sigma(i),\sigma(j) \} \cap \{ k,l \} \neq \emptyset$.

We observe that if $k = \sigma(i)$, then we have $\tau(k) = \sigma \cdot (i,j) \cdot (\sigma(i), \sigma(j)) (k) = j$.
Then, we obtain $j < \sigma(i) = k < \tau(k) = j$, which is absurd.
Similarly, if $k= \sigma(j)$, then we have $\tau(k) = i$, and from $i < \sigma(j) = k < \tau(k) = i$ we obtain another contradiction.

Next observe that if $l = \sigma(i)$, then we have $\tau(l) =  \sigma \cdot (i,j) \cdot (\sigma(i), \sigma(j)) (l) = j$,
and from $j < \sigma(i) < \sigma(j) = l < \tau(l) = j$ we obtain a contradiction. Likewise, $l= \sigma(j)$ is impossible.

If $i=k$, then, of course we must have $j < l$. In this case we must also have that $\tau(k) = \sigma(j)$. In this case,
it is easy to check that $\tau(l) = \sigma(l)$, therefore, $(j,l)$ is an $ee$-rise for $\sigma$ which is bigger than $(i,j)$,
a contradiction.
If $j=k$, then we have $\tau (k) = \sigma(i)$. Just as in the previous case, $(i,l)$ is an $ee$-rise for $\sigma$.
Furthermore, $(i,l) > (i,j)$ gives the contradiction.
Finally, if $j=l$, then it is easy to check that $(k,j)$ is an $ee$-rise for $\sigma$, 
therefore, we have another contradiction, and this finishes the proof of the first case.

\vspace{.75cm}

{\em Case 2:}

We begin with the assumption that $\{ i,j,\sigma(i),\sigma(j)\} \cap \{ k,l \} = \emptyset$.

Then $k>i$. If $k>\sigma( j)$, then observe that $l > \tau(l) = \sigma(l) > \tau(k) = \sigma(k) > k > \sigma(j)$.
It follows that $(k,l)$ is an $ed$-rise for $\sigma$ with a bigger label than $(i,j)$, a contradiction.

We proceed with the assumption that  $i<k<\sigma(j)$.
If $\sigma(k) >j$, then the two moves are interchangeable. If $\sigma(k)$ is in between $i$ and $\sigma(j)$,
then $(k,\sigma(j))$ is an $ee$-rise for $\sigma$.

We proceed with the assumption that $\{ i,j,\sigma(i),\sigma(j)\} \cap \{ k,l \} \neq \emptyset$.

If $k=i$, then we have $j < l$. Since $\tau$ is obtained from $\sigma$ by applying the covering
transformation $(i,j)$, in this case we see that $\tau(k) = \sigma(j)$. Note also that
$\tau(l) = \sigma(l)$. Therefore, $\sigma(j) < \sigma(l)<l$.
If $\sigma(l) <j$, then $(\sigma(j),\sigma(l))$ is an $ee$-rise for $\sigma$ with a label
bigger than $(i,j)$, which is a contradiction.
Otherwise, $(\sigma(j),l)$ is an $ed$-rise for $\sigma$ with a label bigger than $(i,j)$, which is another contradiction.

If $k = j$, then since $\tau$ is obtained from $\sigma$ by the covering transformation of $(i,j)$,
$\tau(j)=\sigma(i)$.
But this is impossible, because $(k,l)$ is an $ed$-rise for $\tau$, and hence $k< \tau(k)$
which implies that $j =k< \sigma(i)$.

If $k= \sigma(i)$, then $\tau(k) = j$ hence $\sigma(j) < j < \tau(l) = \sigma(l)$.
Therefore, $(\sigma(j),l)$ is an $ed$-rise in $\sigma$.

If $k= \sigma(j)$, then $k < \tau(k) = i$, which is absurd.

If $l=j$, then we see that $(k,l)$ is an $ed$-rise for $\sigma$, which is a contradiction.

If $l = \sigma(i)$, then $l> \tau(l) = j > \sigma(i)=l$, which is absurd.

Finally, if $l = \sigma(j)$, then $\tau(k) = \sigma(k)$ and furthermore $\sigma(k)=\tau(k) < \tau(l)= i$.
Therefore, $(k,\sigma(i))$ is an $ed$-rise for $\sigma$, which is a contradiction.

\vspace{.75cm}

{\em Case 3:}

We begin with the assumption that $\{ i,j,\sigma(i),\sigma(j)\} \cap \{ k,l \} = \emptyset$, hence $k > i$.
If $k>j$, then the order of the covering transformations is interchangeable leading to a contradiction.
Therefore we assume that $i<k<j$.
If $\sigma(k)>\sigma(j)$, then once again in this case the two moves are interchangeable.
On the other hand, if $\sigma(i)<\sigma(k)<\sigma(j)$, then $(i,j)$ is not a suitable rise for $\sigma$,
which is a contradiction.

If $\sigma(k) < \sigma(i)$, then we consider two cases; $\sigma(k) > j$ and $\sigma(k) < j$.
In the former case, either the two moves are interchangeable,
or $(k,j)$ is an $ee$-rise for $\sigma$ with a bigger label than $(i,j)$, hence a contradiction.

In the latter case, we have $i< \sigma(k) < j$. In this case, if $l < \sigma(i)$, then the two moves are interchangeable.
If $\sigma(i) < l < \sigma(j)$, then either $\sigma(l)$ is in between $i$ and $j$ or $\sigma(l)$ is greater than $j$.
In the former case, $(i,j)$ is not a suitable rise. If $\sigma(l) > j$, then $(k,l)$ is not a suitable rise for $\tau$,
because in this case $\sigma(k) < \tau(j) = \sigma(i) < \sigma(l)$.
Now, if $\sigma(j) < l$, then we have two possibilities again; either $\sigma(l) > j$ or $\sigma(l)<j$.
In the former case, $(k,l)$ is not a suitable rise for $\tau$. In the latter case, the two moves are interchangeable.

We proceed with the assumption that $\{ i,j,\sigma(i),\sigma(j)\} \cap \{ k,l \} \neq \emptyset$.

If $k=i$ then $(j,l)$ is an $ed$-rise for $\sigma$. Indeed, in this case, $\tau(l) = \sigma(l)$ and we have
the inequalities $j < \sigma(j)=\tau(k) < \tau(l) = \sigma(l) < l$.

If $k=j$ then either $\sigma(l)<\sigma(j)$, or $\sigma(l)>\sigma(j)$. In the former case, we see that
$(i,l)$ is an $ed$-rise for $\sigma$. In the latter case $(j,l)$ is an $ed$-rise for $\sigma$.

If $k= \sigma(i)$, then $k < \tau(k) = \sigma \cdot (i,j) \cdot (\sigma(i), \sigma(j)) (k) = j$.
Since $j < \sigma(i)$, this is a contradiction. Similarly, if $k=\sigma(j)$, then
$k < \tau(k) = \sigma \cdot (i,j) \cdot (\sigma(i), \sigma(j)) (k) = i$. Since $i < \sigma(i)$, this is a contradiction, also.

If $l=j$, then we obtain a contradiction to the facts that $(k,l)$ is an $ed$-rise, and $(i,j)$ is an $ee$-rise.

If $l = \sigma(i)$, then $(k,\sigma(j))$ is an $ed$-rise for $\sigma$, because $k  < \sigma(k) = \tau(k) < \tau(l)= j < \sigma(j)$.

If $l = \sigma(j)$, then $(k,\sigma(i))$ is an $ed$-rise for $\sigma$, because $k  < \sigma(k) = \tau(k) < \tau(l)= i< \sigma(i)$.

\vspace{.75cm}

{\em Case 4:}

We proceed with the assumption that $\{ i,j,\sigma(i),\sigma(j)\} \cap \{ k,l \} = \emptyset$.

Once again, $k>i$. If $k>\sigma(j)$ then the two moves are interchangeable.
Therefore we assume that $i<k<\sigma(j)$.

If $\sigma(k)<\sigma(j)$, then $(k,j)$ is an $ed$-rise for $\sigma$.

If $\sigma(j)<\sigma(k)<j$ then $(k,\sigma(j))$ is an $ee$-rise for $\sigma$.

If $j< \sigma(k)$, then it is easy to check that the two moves are interchangeable.

We proceed with the case that $\{ i,j,\sigma(i),\sigma(j)\} \cap \{ k,l \} \neq \emptyset$.

If $k=i$, then $j < l$. Since $(k,l)$ is an $ee$-rise for $\tau$, we see that $l < \tau(k) = \sigma(j)$,
hence $j < \sigma(j)$. But $(i,j)$ is an $ee$-rise for $\sigma$, hence $j > \sigma(j)$; a contradiction.

If $k=j$, then $i < k < \tau(k) = \sigma \cdot (i,j) \cdot (\sigma(i), \sigma(j)) (k)= \sigma(i)$; a contradiction.

If $k=\sigma(i)$ then either $l>\sigma(j)$, which implies that $(\sigma(j),l)$ is an $ee$-rise for $\sigma$, or
$k<l<\sigma(j)$, which implies that $(i,\sigma(l))$ is an $ed$-rise for $\sigma$ and because $\sigma(l) = \tau(l)$,
the label $(i,\sigma(l))$ is bigger than the label $(i,j)$, hence a contradiction.

If $k=\sigma(j)$, then $i<k< \tau(k) = i$, a contradiction.

If $l=j$, then $i < l <\tau(l) = i$, a contradiction.

Similarly, the case $l=i$ is impossible.

If $l=\sigma(i)$, then we have either $\sigma(k) < \sigma(j)$, or $\sigma(j) < \sigma(k)$.

In the first case, if $\sigma(i) < \sigma(k) <\sigma(j)$, then it is easy to check that $i < k < j$, hence
$(i,j)$ is not a suitable rise. On the other hand, if $\sigma(k) < \sigma(i)$,
we have a contradiction to $\sigma(i)=l < \tau(k)=\sigma(k)$.
We proceed with the case $\sigma(k) > \sigma(j)$, then $(k,\sigma(j))$ is an $ee$-rise for $\sigma$.

If $l = \sigma(j)$, then $\tau(l) = i$ and $i< \sigma(j)$ which is a contradiction.

\vspace{.25cm}

Our next step is to prove that no other chain is lexicographically increasing. Recall that 
our increasing chains are decreasing chains in Incitti's original labeling. 
Since in Incitti's labeling every interval has exactly one decreasing chain the proof is complete.

\end{proof}

\section{Deodhar-Srinivasan poset vs. $(F_{2n},\leq)$}
\label{S:comparison}

As it is mentioned in the introduction, the posets $(\tilde{F}_{2n},\leq_{DS})$ and $(F_{2n},\leq)$ are different.
Indeed, for $2n=6$ the Hasse diagrams of these two posets differ by an edge.

In this section we show that $\tilde{F}_{2n}$ is a subposet of $F_{2n}$.
We proceed by recalling the definition of the length function of $\tilde{F}_{2n}$ as defined in \cite{DeodharSrinivasan}.

Let $[i_1,j_1]\cdots [i_n,j_n]$ be an element from $\tilde{F}_{2n}$, and let $x\in F_{2n}$ denote the corresponding
fixed-point-free involution.
The arc-diagram of $x\in F_{2n}$ is defined as follows. We place the numbers 1 to $2n$ on a horizontal line.
We connect the numbers $i$ and $j$ by a concave-down arc, if $j = x(i)$. Let $c(x)$ denote the number of intersection points of
all arcs.

The length function $\ell_{\tilde{F}_{2n}}$ of $\tilde{F}_{2n}$ is given by
$
\ell_{\tilde{F}_{2n}} ([i_1,j_1]\cdots [i_n,j_n])= \sum_{t=1}^n\left(j_t-i_t-1\right)-c(\pi)\,.
$
See Theorem 1.3 in \cite{DeodharSrinivasan}.

Our first observation is that $\ell_{\tilde{F}_{2n}}$ is in fact an inversion number.
To this end, for $x$ as above, let us define the {\em modified inversion number} of $x$ to be the number of inversions
in the word $i_1 j_1 i_2 j_2 \cdots i_n j_n$, and denote it by $\widetilde{inv}(x)$.
Note that $i_1$ is always 1 for fixed-point-free involutions.

\begin{Proposition}\label{length_functions_equal}
Let $[i_1,j_1]\cdots [i_n,j_n]\in \tilde{F}_{2n}$, and let $x\in F_{2n}$ be the corresponding fixed-point-free involution.
Then
$$
\widetilde{inv}(x) = \ell_{\tilde{F}_{2n}} ([i_1,j_1]\cdots [i_n,j_n]).
$$
\end{Proposition}

\begin{proof}

An inversion in the word $i_1j_1i_2j_2\cdots i_nj_n$ is either the pair $(j_p,i_q)$, or the pair $(j_p,j_q)$, where
$p<q$ and $j_p>i_q$, or $j_p>j_q$, respectively.

We count inversions in another way.
If $(i_t,j_t)$ is a transposition that appears in $[i_1,j_1]\cdots [i_n,j_n]$ of $x$, 
then $j_t-i_t-1=\#\{m\,:\,m\in\mathbb N\,,\,i_t<m<j_t\}$.
On the other hand, each number $m\in \{i_t+1,\dots,j_t-1\}$ appears as an entry 
in another transposition of $[i_1,j_1]\cdots [i_n,j_n]$.

There are three possible cases: (1) the number $m$ is involved in the transposition 
$(a,m)$, where $a<i_t<m$; (2) the number $m$ is involved in the transposition $(a,m)$, 
where $i_t<a<m$; (3) the number $m$ is involved in the transposition $(m,b)$, where $m<b$.

In the first case the pair $(j_t,m)$ is not an inversion.
Notice that when $a<i_t$, the arc corresponding to the transposition $(a,m)$ crosses the arc
corresponding to the transposition $(i_t,j_t)$. In cases 2 and 3, we have the inversion pair $(j_t,m)$ always.
For Case 3, whether $b$ is greater than $j_t$ or not is unimportant.
So, to get the number of inversion pairs $(j_t,*)$ we have to subtract from
$j_t-i_t-1$ the number of intersections of the arc $(i_t,j_t)$ with the arcs $(a,m)$, where $a<i_t<m<j_t$.
Counting the inversions by summing up the contributions of all the transpositions $(i_t,j_t)$ proves our statement.
\end{proof}

Let us illustrate our proof by an example.
\begin{Example}
Take $x =(1,6)(2,5)(3,8)(4,7)\in S_8$.
$$
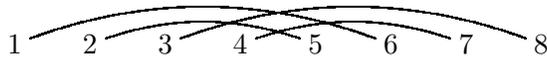
\xygraph{
!{<0cm,0cm>;<1cm,0cm>:<0cm,1cm>::}
!{(1,1) }*+{1}="1"
!{(2,1) }*+{2}="2"
!{(3,1) }*+{3}="3"
!{(4,1) }*+{4}="4"
!{(5,1) }*+{5}="5"
!{(6,1) }*+{6}="6"
!{(7,1) }*+{7}="7"
!{(8,1) }*+{8}="8"
"1"-@/^0.5cm/"6"  "2"-@/^0.3cm/"5" "3"-@/^0.5cm/"8" "4"-@/^0.3cm/"7"}
$$

Start with the transposition $(1,6)$. The numbers between 1 and 6 are 2,3,4,5.
All the pairs $(6,2)$, $(6,3)$, $(6,4)$, $(6,5)$ are inversions of the word $16253847$: 2,3,4 are
involved in transpositions of the form $(m,*)$ which is case (3) in our proof and always gives an inversion,
5 is involved in transposition $(2,5)$, it is case (2) since $1<2$, so it also gives an inversion.
Now take the transposition $(2,5)$. Both of the numbers 3,4 which are between 2 and 5 are involved in transpositions
of case (3), $(3,8)$ and $(4,7)$ and so both of them give inversions $(5,3)$ and $(5,4)$.
Now consider the transposition $(3,8)$. The pair $(8,4)$ is an inversion, 
it is case (3) since 4 is involved in the transposition $(4,7)$.
The pair $(8,7)$ also is an inversion since 7 is involved in the transposition $(4,7)$ and $3<4$, which belongs to case (2).
But the pairs $(8,5)$ and $(8,6)$ are not inversions since 5 and 6 are involved in transpositions $(2,5)$ and $(1,6)$,
where $1<3$ and $2<3$ and so both of them are of case (1). By the same reason when we consider the last transposition of
$x$ which is $(4,7)$, the pairs $(7,5)$ and $(7,6)$ are not inversions, they belong to case (1).
So, summing up, we have four inversions of the form $(6,*)$ contributed by the transposition $(1,6)$,
two inversions of the form $(5,*)$ contributed by the transposition $(2,5)$ and two inversions of the form $(8,*)$
contributed by the transposition $(3,8)$. Thus, $\widetilde{inv} (x)=4+2+2=8$. From the arc diagram depicted above
we see that $c(x)=4$. Hence,
$$
\ell_{\tilde{F}_{2n}}(x)=(6-1-1)+(5-2-1)+(8-3-1)+(7-4-1)-4=4+2+4+2-4=8.
$$
So, we see that $\widetilde{inv} (x)= \ell_{\tilde{F}_{2n}}(x)$ as it is expected.
\end{Example}

\begin{Corollary}
The length functions of $(F_{2n},\leq)$ and $(\tilde{F}_{2n},\leq_{DS})$ are the same.
\end{Corollary}

\begin{proof}
This follows from Proposition~\ref{length_functions_equal} above combined with Proposition 6.2 of \cite{Cherniavsky11}.
\end{proof}

Recall that $y\rightarrow x=[a_1,b_1]\cdots [a_n,b_n]$ in $\tilde{F}_{2n}$, 
if $\ell_{\tilde{F}_{2n}} (y) = \ell_{\tilde{F}_{2n}} (x) +1$
and there exists $1 \leq i<j \leq n$ such that
\begin{enumerate}
\item $y$ is obtained from $x$ by interchanging $b_i$ and $a_j$, where $b_i<a_j$, or
\item $y$ is obtained from $x$ by interchanging $b_i$ and $b_j$,where $b_i < b_j$.
\end{enumerate}
We call these interchanges {\em type 1} and {\em type 2}, respectively. 
Note that, in a type 1 covering relation we have the inequalities
$a_i < b_i < a_j <b_j$.
The inequalities of type 2 are $a_i < a_j< b_i<b_j$.

Note also that an arbitrary interchange of the entries in $x$
does not always result in another element of $\tilde{F}_{2n}$. This is because of the ordering of the $a_i$'s.
For example, as it is seen from Figure 3 of \cite{DeodharSrinivasan}, there is no edge between
the elements $[1,2][3,6][4,5]$ and $[1,4][2,5][3,6]$. On the other hand, it is easy to check that
the corresponding involution $x=(1,2)(3,6)(4,5)$ is covered by $y=(1,4)(2,5)(3,6)$ in $F_{2n}$.

\begin{Theorem}
The covering relations of the poset $\tilde{F}_{2n}$ are among the covering relations of $F_{2n}$.
\end{Theorem}

\begin{proof}

It suffices to note that a type 1 covering relation of $\tilde{F}_{2n}$ corresponds to an $ed$-rise, 
and a type 2 covering relation of $\tilde{F}_{2n}$ corresponds to an $ee$-rise in $F_{2n}$. 
This is almost obvious and straightforward.
\end{proof}

\section{The order complex of $F_{2n}$}
\label{S:ordercomplex}

Concerning the topology of the order complex
$\Delta(F_{2n})$ we have the following theorem.

\begin{Theorem}\label{triangulates}
The order complex $\Delta(F_{2n})$ triangulates a ball of dimension $n(n-1)-2$.
\end{Theorem}

\begin{proof}
We know from \cite{DK74} that if in a pure shellable complex $\Delta$ each $\dim \Delta -1$ dimensional
face lies in at most two maximal faces, then $\Delta$ triangulates a sphere or a ball. 
If there is some codimension one face which is contained in only one maximal face, then the complex is a ball. 

The poset $F_{2n}$ is a shellable (follows from our main result) proper full rank subposet 
(known from Theorem~4.6 of~\cite{H}) of an Eulerian poset which is the interval 
$[j_{2n}, w_0]\subset I_{2n}$ (where $j_{2n}=(1,2)(3,4)\cdots(2n-1,2n)$ and $w_0$ is 
the longest permutation $(1,2n)(2,2n-1)\cdots$ and $I_{2n}$ is Eulerian by~\cite{Incitti04}), 
thus the order complex of the proper part of $F_{2n}$ is a ball of dimension
$$
\dim \Delta(F_{2n}) = \ell (F_{2n}) -2= 2{ n \choose 2}  - 2 = n(n-1)-2.
$$
\end{proof}

\section{Final Remarks}
\label{son}

In Section \ref{S:comparison} we show that the length functions of $F_{2n}$ and $\tilde{F}_{2n}$ are the same.
As a corollary we see that
\begin{Corollary}
For $m\in \N$, let $[m]_q$ denote its $q$-analogue $1+q+\cdots + q^{m-1}$.
Then, the length-generating function $\sum_{x\in F_{2n}} q^{\ell_{F_{2n}}(x)}$ of $F_{2n}$ is equal to
$
[2n-1]_q !! := [2n-1]_q [2n-3]_q \cdots [3]_q [1]_q.
$
\end{Corollary}
\begin{proof}
Follows from Corollary 2.2 of \cite{DeodharSrinivasan}.
\end{proof}
We should mention here that the conclusion of the above corollary is obtained by other combinatorial methods by A. Avni
in his M.Sc. thesis at Bar-Ilan University.

It turns out there is another simple characterization of $\ell_{F_{2n}}$, which seems to be known to the experts.
Although it is not difficult to prove, since we could not locate it in the literature, we record its proof here for
the sake of completeness.
\begin{Proposition}
Let $x\in S_{2n}$ be a fixed-point-free involution.
Then
$
\widetilde{inv}(x)= \ell_{F_{2n}}(x) = \ell_{\tilde{F}_{2n}}(x) = \frac{ inv(x) -n}{2},
$
where $inv(x)= | \{ (i,j):\ i<j\ \text{and}\ x(i)>x(j) \}|$.
\end{Proposition}
\begin{proof}
Let $x\in F_{2n}$.
The second equality is shown to be true in Section \ref{S:comparison}. The first equality follows from
Proposition 6.2 of \cite{Cherniavsky11}. It remains to show the third equality.

Since $F_{2n}$ is a full-rank subposet of a certain interval of $I_{2n}$, and the $I_{2n}$-length 
of the minimal element of $F_{2n}$ is $n$, we have  $\ell_{I_{2n}}(x)-\ell_{F_{2n}}(x)=n$.
On one hand, we know that $\ell_{I_{2n}}(x)=\frac{exc(x)+inv(x)}{2}$, where
$exc(x)=|\{i:i<x(i)\}|$ (see \cite{Incitti04}). On the other hand, if $x\in F_{2n}$, then $exc(x)=n$.
Therefore,
$\ell_{F_{2n}}(x)=\ell_{I_{2n}}(x)-n=\frac{n+inv(x)}{2}-n=\frac{inv(x)-n}{2}$.

\end{proof}



\bibliography{FFI.bib}
\bibliographystyle{plain}

\end{document}